\documentclass[reqno,12pt]{preprint}
\usepackage[marginparwidth=1in,bottom=1.3in]{geometry}
\usepackage[full]{textcomp}
\usepackage[osf]{newtxtext}
\usepackage{cabin}
\usepackage{enumerate}
\usepackage{enumitem}
\usepackage{bbm}
\usepackage{dsfont}

\usepackage[varqu,varl]{inconsolata}
\usepackage[cal=boondoxo]{mathalfa}

\usepackage{comment}
\usepackage{hyperref}
\usepackage{breakurl}
\usepackage{mathrsfs}
\usepackage{booktabs}
\usepackage{caption}

\usepackage{tikz}
\usepackage{amsmath} 
\usepackage{mhequ} 
\usepackage{mhsymb} 
\usepackage{mhenvs} 
\usepackage{microtype}
\usepackage{amssymb} 
\usepackage{eufrak}

\usepackage{wasysym}
\usepackage{centernot}

\usepackage{oldgerm}

\def\E{{\mathbf E}}

\def\N{{\mathbf N}}

\def\R{\mathbf{R}}

\def\T{\mathbf{T}}

\def\Z{\mathbf{Z}}

\def\e{\bold e}

\def\w{\bold w}
\def\x{\bold x}
\def\y{\bold y}
\def\z{\bold z}

\def\St{\mathbb{S}}

\renewcommand{\R}{\mathbb{R}}

\renewcommand{\N}{\mathbb{N}}

\renewcommand{\T}{\mathbb{T}}
\renewcommand{\Z}{\mathbb{Z}}
\renewcommand{\E}{\mathbb{E}}

\newcommand{\cA}{\mathcal{A}}

\newcommand{\cC}{\mathcal{C}}
\newcommand{\cD}{\mathcal{D}}

\newcommand{\cH}{\mathcal{H}}

\newcommand{\cS}{\mathcal{S}}
\let\div\undefined
\DeclareMathOperator{\div}{div}
\newcommand{\1}{\mathds{1}}

\def\s{\mathfrak{s}}

\colorlet{darkblue}{blue!90!black}
\colorlet{darkred}{red!90!black}
\colorlet{darkgreen}{green!50!black}
\colorlet{darkyellow}{yellow!90!black}

\newcommand{\mb}[1]{{\mathbf #1}}
\newcommand{\bb}[1]{{\mathbb #1}}

\makeatletter

\DeclareRobustCommand{\TitleEquation}[2]{\texorpdfstring{\StrLeft{\f@series}{1}[\@firstchar]$\if%
b\@firstchar\boldsymbol{#1}\else#1\fi$}{#2}}

\makeatother

\usepackage{fancyhdr}
\makeatletter
\let\runtitle\@title
\makeatother
\date{%
    \today
    }

\begin{document}

\title{A scaling limit of the 2D parabolic Anderson model with exclusion interaction }

\author{Dirk Erhard$^1$, Martin Hairer$^2$, Tiecheng Xu$^1$}
\institute{Universidade Federal da Bahia, Brazil
    \and EPFL, Switzerland\\%
    \email{erharddirk@gmail.com, 
martin.hairer@epfl.ch,\\
xutcmath@gmail.com}
}

\maketitle

\begin{abstract}
We consider the (discrete) parabolic Anderson model $\partial u(t,x)/\partial t=\Delta u(t,x) +\xi_t(x) u(t,x)$, $t\geq 0$, $x\in \Z^d$. Here, the $\xi$-field is $\R$-valued, acting as a dynamic random environment, and $\Delta$ represents the discrete Laplacian. We focus on the case where $\xi$ is given by a rescaled  symmetric simple exclusion process which converges to an Ornstein--Uhlenbeck process. 
By scaling the Laplacian diffusively and considering the equation on a torus, we demonstrate that in dimension $d=2$, when a suitably renormalized version of the above equation is considered, the sequence of solutions converges in law. This resolves an open problem from~\cite{EH23}, where a similar result was shown in the three-dimensional case. The novel contribution in the present work is the establishment of a gradient bound on the transition probability of a fixed but arbitrary number of labelled exclusion particles.

\end{abstract}

%\tableofcontents

\maketitle

\section{Introduction}
The (discrete) parabolic Anderson model is the partial differential equation
\begin{equ} 
	\label{eq:PAM}
	\partial_t u (x,t)=(\Delta u)(x,t) - \xi_t(x)u(x,t) \;.
\end{equ}
Here $x\in \Z^d$, $t\geq 0$, and $\Delta$ is the discrete Laplacian acting on $u$ as
\begin{equ}[e:delta]
	\Delta u(x,t)= \sum_{y:\, x\sim y}[u(y,t)-u(x,t)],
\end{equ}
where $x\sim y$ means that $x$ and $y$ are nearest neighbours with respect to the Euclidean norm on $\Z^d$, while $\xi$ is an $\R$-valued random field that plays the role of a dynamic random environment. We focus on the case in which $\xi$ is given by the symmetric simple exclusion process, i.e., $\xi=(\xi_t)_{t\geq 0}$ is the Markov process taking values in $\cH=\{0,1\}^{\Z^d}$ with generator $L$ acting on local functions 
$f\in\R^{\cH}$ via
\begin{equation}
	\label{eq:genex}
	(Lf)(\eta) = \sum_{x\sim y}
	[f(\eta^{x,y})-f(\eta)].
\end{equation}
Here, $\eta^{x,y}$ is defined as usual by $\eta^{x,y}(y) = x$, $\eta^{x,y}(x) = y$,
and the identity otherwise.
Informally, \eqref{eq:genex} says that neighbouring states are swapped independently at 
rate~$1$. 
For $\rho\in (0,1)$, denote by $\nu_{\rho}$ the Bernoulli product measure on $\Z^d$ with parameter $\rho$. It is well known that $(\nu_{\rho})_{\rho \in [0,1]}$ forms a family of invariant measures for $\xi$ under which the dynamics is reversible. We assume throughout this article that $\xi_0\sim\nu_\rho$ for some fixed value of $\rho$.
Denote by $\cS(\R^d)$ the space of Schwartz functions and, for $N\in\N$, define the fluctuation field $Y^N$ of $\xi$ via
\begin{equation}
	Y_t^N(f)= 2^{-Nd}\sum_{x\in\Z^d}f(x/2^N)2^{Nd/2}[\xi_{t2^{2N}}(x)-\rho].
\end{equation}
It was shown in \cite{Ravi92} that for all $T>0$ the process $Y^N$ converges in law with respect to the Skorokhod topology in the space $\cD([0,T], \cS'(\R^d))$ to the
stationary generalised Ornstein--Uhlenbeck process $Y$ solving
\begin{equ}[eq:Y]
	\partial_t Y = \Delta Y + \sqrt{2\rho(1-\rho)} \,\div \hat \xi\;,
\end{equ}
where $\hat \xi$ denotes a vector-valued space-time white noise.
%One readily verifies that $Y$ is centred Gaussian with covariance function
%\begin{equation}
%\label{eq:Y}
%\E[Y_t(f)Y_s(g)] = \frac{\rho(1-\rho)}{(4\pi|t-s|)^{d/2}}
%\int_{\R^d}\int_{\R^d} f(u)g(v)\exp\Big(-{\|u-v\|^2 \over 4|t-s|}\Big)\, du\, dv.
%\end{equation} 
This suggests that looking at the equation
\begin{equation}\label{eq:disunrenPAM}
	\partial_t\hat u^N(t,x)=(\Delta_N \hat u^N)(t,x)-2^{Nd/2}\bar\xi^N_{t}(x)\hat u^N(t,x),
\end{equation}
where $\Delta_N =2^{2N}\Delta$ is the rescaled discrete Laplacian acting on functions defined on $2^{-N}\Z^d$, $x\in2^{-N}\Z^d$, $t\geq 0$, $\xi^N_t(x)=\xi_{2^{2N}t}(2^Nx)$, and $\bar \xi^N_t(x)=\xi_{2^{2N}t}(2^Nx)-\rho$ is its centred version, its solutions $\hat u^N$ should converge to the solution $\hat u$ of
\begin{equation}\label{eq:unrenPAM}
	\partial_t \hat u(t,x)=(\Delta \hat u)(t,x)-Y_t(x) \hat u(t,x),\quad x\in\R^d,t\geq 0.
\end{equation}
Here, $\Delta$ denotes the usual continuous Laplacian. It is well known that in all dimensions $Y$ does not make sense as a random space-time function,
but as a random distribution taking values 
in the Besov--H\"{o}lder space $\cC^{-d/2-}(\R^{d+1};\R)$. We therefore see that as the dimension increases the regularity of $Y$ becomes worse. As a consequence, as noted already in~\cite{EH23}, the above equation is ill posed in $d \ge 2$. The problem is that the product 
between $Y$ and $\hat u$ is meaningless in the limit, so that one needs to resort to a renormalisation procedure and rather look at the sequence of equations
\begin{equ} 
	\label{eq:PAMN}
	\partial_t u^{N}(t,x)=\Delta_N u^{N}(t,x) - [2^{Nd/2}\bar\xi^N_{t}(x)-C_N]u^{N}(t,x),
\end{equ}
for a suitable choice of constants $C_N$ that tend to infinity. 
In~\cite{EH23} it was shown that this is indeed the correct approach in $d=3$. Define $\Z_N^d=(\Z/2^N\Z)^d$ and define the rescaled torus by $\T_N^d=2^{-N}\Z_N^d$ . The following result was shown in~\cite{EH23}.
\begin{theorem}
	\label{thm}
	Fix $T>0$. Let $d=3$ and let $(u_{0}^N)_{N\in\N}$ be a sequence of initial conditions such that
	there is $\eta\in (0,1)$ and $u_0\in\cC^\eta$ with
	\begin{equation}
		\lim_{N\to\infty}\|u_0;u_0^N\|_\eta=0\;.
	\end{equation} 
	Write $u^N$ for the solution to~\eqref{eq:PAMN}  defined on $\T_N^d$ 
	and $u$ for the renormalised solution to
	\begin{equ}
	\partial_t u(t,x)= (\Delta u)(t,x) -Y_t(x)u(t,x)
	\end{equ}
	on the three dimensional torus.
	Then, there is a diverging sequence of constants $C_N$ such that the sequence  $u^N$  converges
	in law to $u$ in $\cC_N^{\bar\eta,T}$ for all $\bar\eta \in (0,\frac12\wedge\eta)$. 
\end{theorem}
The spaces of functions and distances used in this statement will be defined in Section~\ref{s:functionspaces} below.
Its proof relies on the (discrete) theory of regularity structures \cite{ErhardHairerRegularity, Hai}, but does not generalise to the two-dimensional setting. This came as a 
surprise, since the theory of regularity structures should be applicable for $d\in\{2,3\}$ and the three-dimensional case should in principle be the harder one. Indeed, in all examples known to the authors
except the one above, the results in low dimensions are a more or less immediate consequences of the 
proof of the result in higher dimension.

The reason for this problem is as follows. The proof in~\cite{EH23} relies on a combination of deterministic and probabilistic arguments. The deterministic arguments are indeed more challenging in higher dimensions but the probabilistic arguments, which involve controlling certain functionals of the exclusion process, become more manageable as the dimension increases.
For the equation above, it was precisely in dimension three that the deterministic and probabilistic arguments worked simultaneously.

A key estimate for the proof of~\ref{thm} is the following result on the joint cumulants of the exclusion process.
\begin{theorem}
	\label{thm:cycle}
	Fix $d\geq 3$.
	Let $\xi$ be the simple symmetric exclusion process with fixed-time distribution that is Bernoulli with parameter $q\in (0,1)$. For $t\geq 0$ and $x\in \T_N^d$ define a rescaled version of $\xi$ via $\xi^N_t(x)= 2^{dN/2}\xi_{2^{2N}t}(2^Nx)$. Then, for every $k \ge 2$ there exists a constant $C$ that does not depend on $N$ such that, for 
	any collection $(t_i, x_i)_{i \in [k]}$ of $k$ space-time points, 
	the joint cumulant satisfies
	\begin{equation}\label{eq:cumubound}
		\E_c \{\xi(t_i,x_i)\}_{i \le k}
		\leq C\sum_{\sigma}\prod_{i \in [k]}(|t_{\sigma(i+1)} - t_{\sigma(i)}|^{\tfrac12}+ |x_{\sigma(i+1)} - x_{\sigma(i)}|\vee 2^{-N})^{-\tfrac{d}{2}},
	\end{equation}
		where the sum runs over all bijections $\sigma \colon [k] \to [k]$
	and we use the convention $\sigma_{k+1}=\sigma_1$.
\end{theorem}

	As noted in~\cite{EH23}, since the right hand side in~\eqref{eq:cumubound} does not depend on $N$, a limiting procedure implies the following result.
\begin{theorem}\label{thm:cumu}
	Let $\xi$ denote the stationary symmetric simple
	exclusion process on $(t,x) \in \R\times \Z^d$ with $d \ge 3$ and $\E\xi(t,x) = \rho \in (0,1)$.
	Then, for every $k \ge 2$ there exists a constant $C$ that does not depend on $N$ such that, for 
	any collection $(t_i, x_i)_{i \in [k]}$ of $k$ space-time points, 
	the joint cumulant satisfies
	\begin{equ}\label{eq:cumu}
		\E_c \{\xi(t_i,x_i)\}_{i \le k} \le C\sum_{\sigma}\prod_{i \in [k]} \bigl(1 + |t_{\sigma_{i+1}} - t_{\sigma_i}|
		+ |x_{\sigma_{i+1}} - x_{\sigma_i}|^2\bigr)^{-d/4}\;,
	\end{equ}
	where the sum runs over all bijections $\sigma \colon [k] \to [k]$
	and we use the convention $\sigma_{k+1}=\sigma_1$.
\end{theorem}

If Theorem~\ref{thm:cycle} were true in dimension $d=2$, then the methods in~\cite{EH23} would be 
sufficient to deduce Theorem~\ref{thm}. 
In this article we show that this is the case, more precisely we show the following:

\begin{theorem}\label{thm:main}
Theorem~\ref{thm} is true in dimension $d=2$ with the same choice of $\eta$ and with $\bar\eta\in(0,1\wedge \eta)$. Theorem~\ref{thm:cycle} is true in dimension $d=2$ as well.
	\end{theorem}
	
	\begin{remark}
Whereas in~\cite{EH23} the constant $C_N$ in Theorem~\ref{thm:main} behaves to leading order as a constant times $2^{2N}$, in $d=2$ it behaves like a multiple of $N$. Indeed, similar to~\cite[Eq.~2.31]{EH23}, the constant is given up to a finite correction by
\begin{equation}
	2^{2N}\int_0^T \sum_{x\in \Z_N^2} (p_{2^{2N}t}(x))^2\, dt = 2^{2N}\int_0^T p_{2^{2N}2t}(0)\, dt\,.
\end{equation}
The claim that this is of order $N$ follows from standard simple random walk estimates. For examples of similar equations with similar renormalisation constants we refer the reader to~\cite[Chapter 10]{Hai} and to~\cite[Section 2]{WongZakai}. However, since the structure of the constant is by now rather standard, 
we do not further comment on why it is of that specific form.
\end{remark}

\subsection{Structure of the article}

The remainder of the article is structured as follows.
In Section~\ref{s:functionspaces} we first introduce the norms and function spaces that appear in Theorem~\ref{thm}. In Section~\ref{s:proof} we then give a conditional proof of Theorem~\ref{thm:main}, 
which is subject to a gradient estimate for the symmetric simple exclusion process. This estimate, 
which constitutes the main result of this article, is introduced and proved in 
Section~\ref{S:gradient} on the full space $\Z^d$ and is adapted to the torus in Section~\ref{s:torus}.

\subsection{Discrete function spaces}\label{s:functionspaces}

For $\eta\in (0,1)$, we define discrete H\"older spaces $\cC_N^\eta(\T_N^d,\R)$ as the space of all elements $f\in\R^{\T_N^d}$, with norm 
\begin{equation}
	\|f\|_{\cC_N^\eta}\overset{\text{def}}{=} \sup_{x\in\T_N^d}|f(x)| + \sup_{x\neq y\in\T_N^d}
	\frac{|f(x)-f(y)|}{|x-y|^\eta}.
\end{equation}
Let $\eta\in (0,1)$. To compare an element $f\in\cC^\eta(\T^d,\R)$ in the usual H\"older space with an element $f^N\in \cC_N^\eta(\T_N^d,\R)$ we introduce the distance
\begin{equation}\label{eq:discreteHolder}
	\begin{aligned}
		\|f;f^N\|_{\eta}\overset{\text{def}}{=}
		&\sup_{x\in\T_N^d}|f(x)-f^N(x)| + \sup_{x\neq y\in\T_N^d}
		\frac{|(f(x)-f(y))-(f^N(x)-f^N(y))|}{|x-y|^\eta}\\
		&+\sup_{\substack{x,y\in\R^d:\, |x-y|<2^{-N}}}
		\frac{|f(x)-f(y)|}{|x-y|^\eta}.
	\end{aligned}
\end{equation}
To compare functions $f\in\cC_\s^\eta ([0,T]\times \T^d,\R)$ and $f^N\colon [0,T]\times \T_N^d \to \R$, we define a ``distance'' by
\begin{equs}
	\|f;f^N\|_{\cC_N^{\eta,T}}&\eqdef \sup_{(t,x)\in [0,T]\times \T_N^d}|f(t,x)-f^N(t,x)|+ \sup_{\substack{(t,x), (s,y)\in [0,T]\times\T_N^d\\ \|(t,x)-(s,y)\|_\s < 2^{-N}}}\frac{|f(t,x)-f(s,y)|}{\|(t,x)-(s,y)\|_\s^\eta}\\
	& + \sup_{\substack{(t,x), (s,y)\in [0,T]\times\T_N^d\\ \|(t,x)-(s,y)\|_\s \geq 2^{-N}}} \frac{|(f(t,x)-f(s,y))-(f^N(t,x)-f^N(s,y))|}{\|(t,x)-(s,y)\|_\s^\eta}.
\end{equs}
Here, the index $\s$ denotes that the space-time distance is the parabolic one, namely
$	\|(t,x)\|_{\s}= \max\{\sqrt{|t|},\|x\|\}$, where $\|\cdot\|$ denotes the usual Euclidean norm.

\section{Proof of the main result}
\label{s:proof}

In this section we prove Theorem~\ref{thm:main} subject to the gradient estimate given in Theorem~\ref{main} below. As already remarked above, it is sufficient to prove Theorem~\ref{thm:cycle} for $d=2$ to adapt the proof of Theorem~\ref{thm} to the two-dimensional setting. Fortunately, most of the work has been done already in~\cite{EH23}. As pointed out in Remarks 1.4 and 3.35 therein, the only missing ingredient to conclude the proof is an estimate on the discrete gradient of the transition probability of $k$ labelled exclusion particles. To describe this bound, we first introduce some notation used in~\cite{EH23}.  Set $\St_N=\Z_N^d=(\Z/2^N \Z)^d$ as well as $\St=\Z^d$.
Let $\cA$ be a fixed finite index set and consider particles labelled by $\cA$  starting from $|\cA|$ 
distinct sites and evolving according to the exclusion rule. More precisely,
the state space of this Markov process is $\St_N^{\cA}$, defined by
\begin{equ}
	\St_N^{\cA}= \{\x = (\x_a)_{a\in\cA}:\,\x_a\in\St_N\text{ for all }a\in\cA,\, \x_a\neq \x_b\, \text{ for all } a\neq b\}\;,
\end{equ}
and its generator is given by
\begin{equation}
	\label{eq:labelledexclusion}
	(L_\cA f)(\x) = \sum_{\{x,y\}}[f(\sigma^{x,y}\x)-f(\x)]\;,
\end{equation}
where the sum runs over unoriented bonds $\{x,y\}$ between any pair of neighbouring sites $x,y\in\St_N$.
Here, for $\x\in\St_N^\cA$ the configuration $\sigma^{x,y}\x$ is given by
\begin{equation}
	\label{eq:sigmaswapped}
	(\sigma^{x,y}\x)_a = 
	\left\{\begin{array}{ll}
		y, &\mbox{if }\x_a= x,\\
		x, &\mbox{if }\x_a=y,\\
		\x_a, &\mbox{otherwise}.
	\end{array}\right.
\end{equation}
We also introduce $\St^{\cA}$, defined analogously to $\St_N^{\cA}$.
We further write  $p_t^{\ell}(\x,\y)$ for the probability that the labelled exclusion process 
starting at $\x$ at time $0$ is at $\y$ at time $t$.

Given $\x \in \St_N^A$ and $\bar A \subset A$,
we write $\x_{\bar A} \in \St_N^{\bar A}$ for the restriction of $\x$ to $\bar A$.
Given two disjoint index sets $A, B$ and $t =(t_1, t_2)\in \R_+^2$, we also write
\begin{equ}
	P_t^{A,B}(\x,\y) = p_{t_1}^{\ell}(\x_A,\y_A)p_{t_2}^{\ell}(\x_B,\y_B)\;,\qquad \x,\y \in \St_N^{A\sqcup B}\;.
\end{equ}
Given furthermore two elements $i\in A$ and $j \in B$, we write
\begin{equ}[e:grad]
	\nabla_{i,j} P_t^{A,B}(\x,\y) = 
	\big(p_{t_1}^{\ell}(\x_A^{ij},\y_A) - p_{t_1}^{\ell}(\x_A,\y_A)\big)\big(p_{t_2}^{\ell}(\x_B^{ij},\y_B)-p_{t_2}^{\ell}(\x_B,\y_B)\big)\;,
\end{equ}
if $\|\x_i - \x_j\| = 1$ and $\nabla_{i,j} P_t^{A,B}(\x,\y) = 0$ otherwise. Here, we wrote
$\x^{ij}$ as a shorthand for $\sigma^{i,j} \x$.
With these notations, by Remarks 1.4 and 3.35 in~\cite{EH23} Theorem~\ref{thm:main} follows if we can show that, for $d=2$, there exist positive constants $\theta$ and $C=C(\theta, |A|, |B|)$ such that uniformly in $N$ for all $\x,\y\in \St_N^{A\sqcup B}$
and all $t\geq 0$,
\begin{equs}
		 |\nabla_{i,j} P_{t}^{A,B}(\x,\y)|	\leq
		&C(\sqrt{|t_1|} +\|\x_i-\y_i \| + 2^{-N})^{-\theta} (\sqrt{|t_2|} +\|\x_j-\y_j\| + 1)^{-\theta}\\
		& \times \prod_{k\in A\sqcup B}(\sqrt{|t_k|} +\|\x_k-\y_k\| + 1)^{-d}\;,
\end{equs}
where $t_k=t_1$ if $k\in A$ and $t_k=t_2$ if $k\in B$. 
By the multiplicative structure of $\nabla_{i,j} P_t^{A,B}$ it suffices to show that
\begin{equ}\label{eq:toshow}
	\begin{aligned}
|p_{t}^\ell(\x_A^{ij},\y_A) - p_{t}^\ell(\x_A,\y_A)|\lesssim (\sqrt{|t|} +\|\x_i-\y_i\| + 1)^{-\theta} \prod_{k\in A}(\sqrt{|t|} +\|\x_k-\y_k\| + 1)^{-d} \,.
\end{aligned}
\end{equ}
Note that the components of $\x$ and $\y$ indexed by $B$ do not play any role anymore. Moreover, since the estimates we are after do not depend on the structure of $A$ but only on its cardinality we can assume throughout the remainder of the article that $A$ is of the form $\cA=A=\{1,2,\ldots, k\}$ for some positive integer $k$. Furthermore, we can assume that $\x_A^{ij}$ and $\x_A$ only differ in the first component of their first coordinate. In the sequel we write $\St_N^{\cA}= \St_N^k$ and similarly for $\St^{\cA}$. 
Equation~\eqref{eq:toshow} and therefore Theorem~\ref{thm:cycle} will then follow from the following result, which is stated for all dimensions.
\begin{theorem}\label{thm:gradboundtorus}
Let $d\in \N$, $k\geq 1$, and $\theta\in(0,1)$. There exists a universal constant $C=C(\theta,k)$ such that
\begin{equ}[e:mainGradientBound]	
	\big\lvert p_{t}^{\ell}(\x,\y)\,-\,p_{ t}^{\ell}(\x+\e_{11},\y)\big\rvert\,\leq\, \frac{C}{(\sqrt{t}\,+\,\|\x-\y\|+1)^{kd+\theta}},
\end{equ}
	valid for all $\x,\y,\x+\e_{11}\in\St_N^k$, $N\in \N$ and $t\geq 0$. 
	In the above formula,
	\begin{equ}\e_{11}\,=\,((1,0,\cdots,0),(0,\cdots,0),\cdots,(0,\cdots,0))\in (\Z^d)^k.
	\end{equ}
\end{theorem}  
Since the above result might come in handy in a context different from ours, we will prove it first for the exclusion process defined on all of $\mathbb{Z}^d$ in Section~\ref{S:gradient}. The proof of Theorem~\ref{thm:gradboundtorus} then follows from a simple adaptation of the arguments given in Section~\ref{S:gradient} and will be presented in Section~\ref{s:torus}.

\section{Gradient estimate}\label{S:gradient}
In this section we prove Theorem~\ref{thm:gradboundtorus}. We will however prove it first on the full space and then explain how to adapt its proof to the torus. With a slight abuse of notation we will from now on denote by $p^{\ell}$ the transition probability for the labelled exclusion process on $\Z^d$ (more precisely on $\St^k$ for some given integer $k$).
\begin{theorem}\label{main}
The statement of Theorem~\ref{thm:gradboundtorus} holds with $\St_N^k$
replaced by $\St^k$.
\end{theorem} 

\begin{remark}
	It is not completely clear to us what the optimal bound of this type should be. A naive guess based on the upper bound for $p^\ell$ obtained in~\cite{Landim05} would be that the optimal bound is \eqref{e:mainGradientBound} but with
	$\theta = 1$.
\end{remark}
\begin{remark}\label{rem:d=1}
	The above result was obtained in the one dimensional case in \cite{EFX23}. Therefore, for the remainder of the proof we will assume that $d\geq 2$.
\end{remark}
The proof of this theorem will be given in the forthcoming section. The proof in the one-dimensional case in~\cite{EFX23} uses a coupling argument between two exclusion processes started at different positions.  It is not clear to us how to use a similar argument in the higher dimensional case, which prompts us to use a different strategy. To be more precise, inspired by~\cite{Andjel2013}, we bound the difference between the transition probability of a labelled SSEP and of a system of independent random walks. Then it only remains to establish a gradient estimate of the transition probability of independent random walks.
\begin{remark}
	Throughout this article,  $C$ denotes a constant that depends only on fixed parameters and may change from line to line. However constants with subscript such as $C_i$ will be fixed throughout the article.
	\end{remark}

\subsection{Proof of Theorem~\ref{main}}\label{S:step}
The proof of Theorem~\ref{main} is based on the following spatially uniform gradient estimate.
\begin{theorem}\label{unifbound}
Let $d\in\N$.	Fix $k\geq 1$ and $\theta\in(0,1)$. There exists a universal constant $C=C(\theta,k)$ such that
	$$\big\lvert p_t^\ell(\x,\y)\,-\,p_t^\ell(\x+\e_{11},\y)\big\rvert\,\leq\, \frac{C}{(\sqrt{t}\,+\,1)^{kd+\theta}},$$
	valid for all $\x,\y,\x+\e_{11}\in\St^k$, $t\geq 0$.
\end{theorem} 
We first show how Theorem~\ref{main} follows from Theorem~\ref{unifbound}.
By the trivial bound 
$$\big\lvert p_t^\ell(\x,\y)\,-\,p_t^\ell(\x+\e_{11},\y)\big\rvert\,\leq\,p_t^\ell(\x,\y)\,+\,p_t^\ell(\x+\e_{11},\y)$$
and \cite[Theorem 1.1]{Landim05}, there exist some positive constants $C_1,C_2$ such that
\begin{equation}\label{landimbound}
\big\lvert p_t^\ell(\x,\y)\,-\,p_t^\ell(\x+\e_{11},\y)\big\rvert\,\leq\, \frac{C_1}{(\sqrt{t}\,+\,1)^{kd}}\exp\Big\{ -\frac{C_2t}{2(\log t)^2}\Phi\Big( \frac{\|\x-\y\| \log t}{C_2^2 t} \Big)\Big\}
\end{equation}
which is valid for all $t>0$, where 
$$\Phi(u)\,:=\,\sup_{w\in\bb R}\big( uw- w^2\cosh w\big)\;.$$ 
Fix $\theta\in(0,1)$ and choose $\theta'=1-\frac{1}{2}(1-\theta)$. Applying the geometric interpolation inequality
$\min\{a,b\}\,\leq\, a^\gamma b^{1-\gamma}$
with $\gamma$ such that 
$$(1-\gamma) kd+\gamma(kd+\theta')\,=\,kd+\theta\;,$$
$a=C(\sqrt{t}+1)^{-kd-\theta'}$, and $b$ equal to the right hand side in \eqref{landimbound}, we have
\begin{equation*}
\big\lvert p_t^\ell(\x,\y)\,-\,p_t^\ell(\x+\e_{11},\y)\big\rvert\,\leq\, \frac{C_3}{(\sqrt{t}\,+\,1)^{kd+\theta}}\exp\Big\{ -\frac{C_4t}{(\log t)^2}\Phi\Big( \frac{\|\x-\y\| \log t}{C_2^2 t} \Big)\Big\}
\end{equation*}
for some positive constants $C_3, C_4$ depending on $C_1, C_2, \theta, d$ and $k$. As in the
proof of \cite[Lemma~3.5]{EH23}, we can use the fact that $\Phi(u) \eqsim u^2$ for $u \le 1$ and $\Phi(u) \gtrsim u$ for $u \ge 1$ to bound the right hand side by 
$$\frac{C}{(\sqrt{t}+\|\x-\y\|+1)^{kd+\theta}}\;.$$
This concludes the proof of Theorem~\ref{main}. 

To prove Theorem \ref{unifbound}, we use the semigroup property:
\begin{equation*}
\begin{split}
\big\lvert  p^{\ell}_{2t}(\x,\y)\,-\,p^{\ell}_{ 2t}(\x+\e_{11},\y)\big\rvert\,=&\,\sum_{\z\in\St^k}\Big\lvert \big[p_t^\ell(\x,\z)\,-\,p_t^\ell(\x+\e_{11},\z) \big] p_t^\ell(\z,\y) \Big\rvert\\
\,\leq\,&\sum_{\z\in\St^k} \big\lvert p_t^\ell(\x,\z)\,-\,p_t^\ell(\x+\e_{11},\z) \big\rvert \sup_{\w}p_t^\ell(\w,\y).
\end{split}
\end{equation*}
It follows from~\cite[Theorem 1.1]{Landim05} that 
$ \sup_{\w}p_t^\ell(\w,\y)\,\ \leq\, C (\sqrt{t}+1)^{-kd}$, so
that it remains to show that
\begin{equation}\label{TV}
\sum_{\z\in\St^k} \big\lvert p_t^\ell(\x,\z)\,-\,p_t^\ell(\x+\e_{11},\z) \big\rvert\,\leq\,\frac{C}{(\sqrt{t}+1)^\theta}.
\end{equation}
The rest of this section is dedicated to this estimate, which is a consequence of 
Proposition~\ref{sumgradrw} and Lemma~\ref{sumy} below.

\subsubsection{Estimates of independent random walks}

 Write $p_t^{rw}$ for the transition probability of $k$ independent random walks evolving on $\bb Z^d$, so that
\begin{equation}\label{ptrwdec}
	p_t^{rw}((\x_1,\cdots,\x_k),(\y_1,\cdots,\y_k))\,=\,\prod_{i=1}^k p_t(\x_i,\y_i)\;,
\end{equation}
where $p_t$ is the transition probability of the $d$-dimensional simple random walk. 
\begin{proposition}\label{sumgradrw}
There exists a constant $C=C(k,d)$ such that
\begin{equation*}
\sum_{\z\in (\bb Z^d)^k} \big\lvert p_t^{rw}(\x,\z)\,-\,p_t^{rw}(\x+\mb e_{11},\z) \big\rvert\,\leq\, \frac{C}{\sqrt{t}+1}.
\end{equation*}
\end{proposition}
\begin{proof}
By \eqref{ptrwdec}, this expression equals
\begin{equ}
\sum_{\z_1\in\bb Z^d} \big\lvert p_t(\x_1,\z_1)\,-\,p_t(\x_1+e_1,\z_1) \big\rvert \prod_{i\neq 1}\sum_{\z_i\in\bb Z^d}p_t(\x_i,\z_i)
= \sum_{\z_1\in\bb Z^d} \big\lvert p_t(\x_1,\z_1)\,-\,p_t(\x_1+\e_1,\z_1) \big\rvert,
\end{equ}
where $\e_1$ is the first unit vector in $\bb Z^d$.
It follows from \cite[Theorem~2.3.6]{LawlerLimic} and a large deviation estimate that 
$$\big\lvert p_t(\x_1,\z_1)\,-\,p_t(\x_1+\e_1,\z_1) \big\rvert\,\leq\, \frac{C}{(\sqrt{t}+\|\x_1-\z_1\|+1)^{d+1}}.$$
Summing the upper bound on the right hand side over all $\z_1$, we obtain
the desired bound, thus completing the proof.
\end{proof}

\subsubsection{Comparison between \TitleEquation{p_{t}^{\ell}}{p_t^ell} and 
\TitleEquation{p_{t}^{rw}}{p_t^rw}}\label{sec:comp}

An integration by parts formula gives 
\begin{equation}\label{formula}
	p^{rw}_{t}(\mb x,\mb y)\,-\,p_t^\ell(\mb x,\mb y)\,=\,\int_0^t P_{t-s}^{\ell}(L^{rw}-L^{\ell})( P_s^{rw} \delta_{\mb y})(\mb x) \,ds\;,
\end{equation}
where $L^{rw}$ and $L^{\ell}$ are the generators of the independent random walks and the labelled exclusion process respectively and $P_t^{rw}$ / $P_t^{\ell}$ are the semigroups they generate.
 These operators act on local functions $f:(\bb Z^d)^k\to\bb R$ by
\begin{equs}
L^{\ell}f(\mb x) = &\sum_{\mb y:\|\mb y-\mb x\|=1}[f(\mb y)-f(\mb x)]\1\{\mb y\in \St^k\} \\
 &+\sum_{1\leq i< j\leq k}\1\{\|\x_i-\x_j\|=1\} [f(\sigma^{i,j}\mb x)-f(\mb x)]\1\{\mb x\in \St^k\}\,,\\
L^{rw}f(\mb x) =&\sum_{\mb y:\|\mb y-\mb x\|=1}[f(\mb y)-f(\mb x)]\,.
\end{equs}
Note that we have extended $L^\ell$ to act on functions on all of $(\bb Z^d)^k$ and not just
$\St^k$. Our specific choice of extension is natural, but other choices are possible.

The first expression is indeed equal to~\eqref{eq:labelledexclusion} when acting on functions 
that vanish outside $\St^k$ since the indicator restricts to those configurations such that 
no two particles can ever occupy the same site.

\begin{lemma}\label{expression}
	For every $\mb x,\mb y\in\St^k$,
	\begin{equs}
			p^{rw}_{t}&(\mb x,\mb y)\,-\,p_t^\ell(\mb x,\mb y)\\
&= \sum_{\w\in\St^{k}}\int_0^t  p_{t-s}^{\ell}(\mb x,\mb w)\sum_{i\neq j} \1\{\|\w_j-\w_i\|= 1\}[p_s^{rw}(\delta^{i,j}\mb w,\mb y)-p_s^{rw}(\mb w,\mb y)] ds\\
&\quad -\sum_{\w\in\St^{k}}\int_0^t  p_{t-s}^{\ell}(\mb x,\mb w)\sum_{ i<j} \1\{\|\w_i-\w_j\|=1\}[p_s^{rw}(\sigma^{i,j}\mb w,\mb y)-p_s^{rw}(\mb w,\mb y)]ds\,,
	\end{equs}
where $\delta^{i,j}\w\in(\bb Z^d)^k$ is defined by
\begin{equation*}
(\delta^{i,j}\w)_\ell=\w_\ell, \,\forall\,\,\ell\neq j, \quad \text{and} \;\;  (\delta^{i,j}\w)_j=\w_i. 
\end{equation*}
\end{lemma}
\begin{proof}
Let $\mb w\in \St^k$, a straightforward computation shows that
	\begin{equation}\label{gendif}
		\begin{aligned}
			(L^{rw}-L^{\ell})g(\mb w)\,=\, &\sum_{\mb z:\|\mb z-\mb w\|=1}[g(\mb z)-g(\mb w)]\1\{\mb z\notin\St^k\}\\
&-\,\sum_{1\leq i< j\leq k}\1\{\|\mb w_i-\mb w_j\|=1\} [g(\sigma^{i,j}\mb w)-g(\mb w)]\,.
		\end{aligned}
	\end{equation}
Observe that $\mb w\in\St^k$, and therefore $\w_i\neq \w_j$ for all $i\neq j$. If $\mb z$ has distance $1$ to $\mb w\in\St^k$, then there exists at most one pair of coordinates $i\neq j$ such that $\z_i=\z_j$. In the case of existence of such pair, we must have $\|\w_i-\w_j\|=1$, and this is the only
way in which one can have $\mb \z\notin \St^k$. Therefore the first term on the right hand side of \eqref{gendif} vanishes unless there exists a unique pair of coordinates $i\neq j$ such that $\|\w_j-\w_i\|=1$ and such that
$\z_\ell=\w_\ell$ for all $\ell\neq j$, and $\z_j=\w_i$,
in other words such that $\z = \delta^{i,j}\w$.
We conclude that, if $\mb w\in\St^k$, then 
\begin{equ}
	\sum_{\mb z:\|\mb z-\mb w\|=1}[g(\mb z)-g(\mb w)]\1\{\mb z\notin\St^k\}
		= \sum_{i\neq j} \1\{\|\w_j-\w_i\|=1\}[g(\delta^{i,j}\mb w)-g(\mb w)]\;.
\end{equ}
This equation allows us to deduce that, if $\mb w\in\St^k$, 
\begin{equs}
(L^{rw}-L^{\ell})( P_s^{rw} f)(\mb w) 
& = \sum_{i\neq j} \1\{\|\w_j-\w_i\|=1\}\sum_{\z\in(\bb Z^d)^k}[p_s^{rw}(\delta^{i,j}\mb w,\mb z)-p_s^{rw}(\mb w,\mb z)]f(\mb z)\\
&-\sum_{i< j}\1\{\|\mb w_i-\mb w_j\|=1\} \sum_{\z\in (\Z^d)^k}[p_s^{rw}(\sigma^{i,j}\mb w,\z)-p_s^{rw}(\mb w,\z)]f(\z)\,.
\end{equs}
Choosing $f=\delta_{\mb y}$, the claim follows from this equation combined with \eqref{formula} and the fact that $p^\ell$ preserves $\St^k$.
\end{proof}

Before we formulate the next lemma we recall that by Remark~\ref{rem:d=1} we can assume that $d\geq 2$.
\begin{lemma}\label{sumy}
There exists a constant $C=C(k)$ such that for any $\x\in\St^k$
	\begin{equation}
		\sum_{\mb y\in\St^k}\big\lvert p^{rw}_{t}(\mb x,\mb y)\,-\,p_t^\ell(\mb x,\mb y)\big\rvert \,\leq\,\begin{cases}
\frac{C\log(t+2)}{\sqrt{t}+1} & \text{if} \,\, d=2,\\
\frac{C}{\sqrt{t}+1}  & \text{if} \,\, d\geq 3.\\
\end{cases}
	\end{equation}
\end{lemma}
\begin{proof}
First of all, as a simple consequence of Proposition \ref{sumgradrw} and the triangle inequality, there exists a universal constant $C$ independent of $\mb w$ such that, for every $\mb w$ with $\|\w_i-\w_j\|=1$ for some $i<j$,%\noteDe{Isnt this a direct consequence of Proposition~\ref{sumgradrw}?}
\begin{equation}\label{sumz1}
\sum_{\mb y\in  (\bb Z^d)^k} \big|p_s^{rw}(\bar{\mb w},\mb y)-p_s^{rw}(\mb w,\mb y)\big|\,\leq\, \frac{C}{\sqrt{s}+1},\qquad \bar{\mb w}\in\{\delta^{i,j}\mb w,\delta^{j,i}\mb w,\sigma^{i,j}\mb w\}\;.
\end{equation}

%Since these two estimate can be proved in the same way, we only give the proof of \eqref{claim1}.  Recall that $p_s(\cdot,\cdot)$ is the transition probability of a single random walker at time $s$. The expression on the left hand side of \eqref{claim1} can be written as 
%\begin{equation*}
%\sum_{\mb y\in  (\bb Z^d)^k} \big\lvert  p_s(w_i,y_j)-p_s(w_j,y_j)     \big\rvert\prod_{\ell\neq j}p_s(w_\ell,y_\ell).
%\end{equation*}
%Using that
%$$\sum_{y_\ell\in\bb Z}p_s(w_\ell,y_\ell)=1,$$
%the previous expression is equal to
%\begin{equation*}
%\sum_{y_j\in\bb Z^d} \big\lvert  p_s(w_i,y_j)-p_s(w_j,y_j)     \big\rvert.
%\end{equation*}
%Since $|w_i-w_j|=1$, the claim follows from Proposition  \ref{sumgradrw}.

We now show that there exists a constant $C=C(k)$ such that
\begin{equation}\label{claim1}
\sum_{\w\in\St^{k}}  p_{t-s}^{\ell}(\mb x,\mb w)\sum_{i<j} \1\{\|\w_j-\w_i\|= 1\}\,\leq\, \frac{C}{(\sqrt{t-s}+1)^{d}}.
\end{equation}
Indeed, by Lemma 3.5 of \cite{EH23}, we can bound the left hand side of \eqref{claim1} by 
\begin{equs}
\sum_{\w\in\St^{k}} & \frac{C}{(\sqrt{t-s}+\|\x-\w\|+1)^{kd}}\sum_{i<j} \1\{\|\w_j-\w_i\|= 1\}\\
&\leq\sum_{i<j}\frac{C}{(\sqrt{t-s}+1)^{d-1}}\sum_{\w\in\St^{k}}  \frac{ \1\{\|\w_j-\w_i\|= 1\}}{(\sqrt{t-s}+\|\x-\w\|+1)^{(k-1)d+1}}.
\end{equs}
A straightforward computation shows that the second sum in the last line is bounded by
$\frac{C}{\sqrt{t-s}+1}$, which proves \eqref{claim1}.

 In view of Lemma \ref{expression}, by \eqref{sumz1} and \eqref{claim1},  it remains to show that
\begin{equation}
\int_0^t \frac{1}{(\sqrt{t-s}+1)^d}\frac{1}{\sqrt{s}+1} ds\,\leq\, \begin{cases}
\frac{C\log (t+2)}{\sqrt{t}+1} & \text{if} \,\, d=2,\\
\frac{C}{\sqrt{t}+1}  & \text{if} \,\, d\geq 3.\\
\end{cases}
\end{equation}

We deal with this integral by dividing the integration interval into $[0,\frac{t}{2})$ and $[\frac{t}{2},t]$. Bounding $\frac{1}{\sqrt{t-s}+1}$ on $[0,\frac{t}{2})$ from above by $\frac{1}{\sqrt{t/2}+1}$, since $d\geq 2$, we have
\begin{equs}
\int_0^{\frac{t}{2}} \frac{1}{(\sqrt{t-s}+1)^d}\frac{1}{\sqrt{s}+1} ds
&\leq\frac{C}{(\sqrt{t}+1)^d}\int_0^{\frac{t}{2}}\frac{1}{\sqrt{s}+1} ds\\
&\leq\frac{C\sqrt{t}}{(\sqrt{t}+1)^d}\,\leq\,  \frac{C}{\sqrt{t}+1}. 
\end{equs}
Bounding $\frac{1}{\sqrt{s}+1}$ on $[\frac{t}{2},t]$ from above by $\frac{1}{\sqrt{t/2}+1}$ we obtain
\begin{equ}
\int_{\frac{t}{2}}^t \frac{1}{(\sqrt{t-s}+1)^d}\frac{1}{\sqrt{s}+1} ds
\leq \frac{C}{\sqrt{t}+1}\int_{\frac{t}{2}}^t\frac{1}{(\sqrt{t-s}+1)^d} ds\;,
\end{equ}
which is of the desired order of magnitude, thus concluding the proof.
\end{proof}

\subsection{Adaptation of Theorem~\ref{main} to the torus}\label{s:torus}
In this very short section we explain how to adapt the proof of Theorem~\ref{main} to the torus. More precisely, how to obtain Theorem~\ref{thm:gradboundtorus}. Analysing carefully the proof of Theorem~\ref{main} we note that it depends crucially on the bounds obtained in~\cite{Landim05} and on estimates on the transition probability of a simple random walk. The latter can be adapted without any major problems to the torus, whereas the adaptation of the results in~\cite{Landim05} has been done in~\cite{EH23}. We therefore can conclude the proof of Theorem~\ref{thm:gradboundtorus}.

\section*{Acknowledgement}
D.E.\ was supported by the National Council for Scientific and Technological Development -- CNPq via a Bolsa de Produtividade 303348/2022-4 and via a Universal Grant (Grant Number 406001/2021-9). D.E.\ and T.X.\ moreover acknowledge support by the Serrapilheira Institute (Grant Number Serra-R-2011-37582).

\bibliography{refs}

\def\cprime{$'$} \def\polhk#1{\setbox0=\hbox{#1}{\ooalign{\hidewidth
  \lower1.5ex\hbox{`}\hidewidth\crcr\unhbox0}}}
\begin{thebibliography}{EFX23}
\expandafter\ifx\csname url\endcsname\relax
  \def\url#1{\texttt{#1}}\fi
\expandafter\ifx\csname urlprefix\endcsname\relax\def\urlprefix{URL }\fi
\expandafter\ifx\csname href\endcsname\relax
  \def\href#1#2{#2}\fi
\expandafter\ifx\csname burlalt\endcsname\relax
  \def\burlalt#1#2{\href{#2}{\texttt{#1}}}\fi

\bibitem[And13]{Andjel2013}
\textsc{E.~D. Andjel}.
\newblock Finite exclusion process and independent random walks.
\newblock \emph{Braz. J. Probab. Stat.} \textbf{27}, no.~2, (2013), 227--244.
\newblock
  \burlalt{doi:10.1214/11-BJPS170}{http://dx.doi.org/10.1214/11-BJPS170}.

\bibitem[EFX23]{EFX23}
\textsc{D.~{Erhard}}, \textsc{T.~{Franco}}, and \textsc{T.~{Xu}}.
\newblock {Non-equilibrium joint fluctuations for the current and occupation
  time in the symmetric exclusion process}.
\newblock \emph{arXiv e-prints} (2023).
\newblock \burlalt{arXiv:2304.13790}{http://arxiv.org/abs/2304.13790}.

\bibitem[EH19]{ErhardHairerRegularity}
\textsc{D.~Erhard} and \textsc{M.~Hairer}.
\newblock Discretisation of regularity structures.
\newblock \emph{Ann. Inst. Henri Poincar\'{e} Probab. Stat.} \textbf{55},
  no.~4, (2019), 2209--2248.
\newblock \burlalt{arXiv:1705.02836}{http://arxiv.org/abs/1705.02836}.
\newblock
  \burlalt{doi:10.1214/18-AIHP947}{http://dx.doi.org/10.1214/18-AIHP947}.

\bibitem[EH24]{EH23}
\textsc{D.~Erhard} and \textsc{M.~Hairer}.
\newblock A scaling limit of the parabolic {A}nderson model with exclusion
  interaction.
\newblock \emph{Comm. Pure Appl. Math.} \textbf{77}, no.~2, (2024), 1065--1125.
\newblock \burlalt{arXiv:2103.13479}{http://arxiv.org/abs/2103.13479}.
\newblock \burlalt{doi:10.1002/cpa.22145}{http://dx.doi.org/10.1002/cpa.22145}.

\bibitem[Hai14]{Hai}
\textsc{M.~Hairer}.
\newblock A theory of regularity structures.
\newblock \emph{Invent. Math.} \textbf{198}, no.~2, (2014), 269--504.
\newblock \burlalt{arXiv:1303.5113}{http://arxiv.org/abs/1303.5113}.
\newblock
  \burlalt{doi:10.1007/s00222-014-0505-4}{http://dx.doi.org/10.1007/s00222-014-0505-4}.

\bibitem[HP15]{WongZakai}
\textsc{M.~Hairer} and \textsc{E.~Pardoux}.
\newblock A {W}ong-{Z}akai theorem for stochastic {PDE}s.
\newblock \emph{J. Math. Soc. Japan} \textbf{67}, no.~4, (2015), 1551--1604.
\newblock \burlalt{arXiv:1409.3138}{http://arxiv.org/abs/1409.3138}.
\newblock
  \burlalt{doi:10.2969/jmsj/06741551}{http://dx.doi.org/10.2969/jmsj/06741551}.

\bibitem[Lan05]{Landim05}
\textsc{C.~Landim}.
\newblock Gaussian estimates for symmetric simple exclusion processes.
\newblock \emph{Ann. Fac. Sci. Toulouse Math. (6)} \textbf{14}, no.~4, (2005),
  683--703.
\newblock \burlalt{arXiv:math/0505089}{http://arxiv.org/abs/math/0505089}.

\bibitem[LL10]{LawlerLimic}
\textsc{G.~F. Lawler} and \textsc{V.~Limic}.
\newblock \emph{Random walk: a modern introduction}, vol. 123 of
  \emph{Cambridge Studies in Advanced Mathematics}.
\newblock Cambridge University Press, Cambridge, 2010.
\newblock
  \burlalt{doi:10.1017/CBO9780511750854}{http://dx.doi.org/10.1017/CBO9780511750854}.

\bibitem[Rav92]{Ravi92}
\textsc{K.~Ravishankar}.
\newblock Fluctuations from the hydrodynamical limit for the symmetric simple
  exclusion in {${\bf Z}^d$}.
\newblock \emph{Stochastic Process. Appl.} \textbf{42}, no.~1, (1992), 31--37.
\newblock
  \burlalt{doi:10.1016/0304-4149(92)90024-K}{http://dx.doi.org/10.1016/0304-4149(92)90024-K}.

\end{thebibliography}
\bibliographystyle{Martin}

\end{document}